\documentclass{article}
\usepackage{geometry}                % See geometry.pdf to learn the layout options. There are lots.
\usepackage{graphicx}
\usepackage{amssymb}
\usepackage{epstopdf}
\usepackage{theorem}
{\renewcommand{\theththm}{\ref{#1}}\begin{ththm}}
{\end{ththm}}
\newtheorem{ththm}{Theorem}

\usepackage{amsmath,amssymb}  % Better maths support & more symbols
\usepackage{bm}  % Define \bm{} to use bold math fonts
\usepackage{color}
\usepackage{xypic}
\usepackage[applemac]{inputenc}
\usepackage{mathtools}
\usepackage[pdftex,bookmarks,colorlinks,breaklinks]{hyperref}  % PDF hyperlinks, with coloured links
\definecolor{dullmagenta}{rgb}{0.4,0,0.4}   % #660066
\definecolor{darkblue}{rgb}{0,0,0.4}
\hypersetup{linkcolor=black,citecolor=black,filecolor=black,urlcolor=black} % black links, for printed output

\usepackage{memhfixc} 

\newtheorem{theorem}{Theorem}[section]
\newtheorem{proposition}[theorem]{Proposition}

\newtheorem{cor}[theorem]{Corollary}

{\theorembodyfont{\normalfont\rmfamily}

} \makeatletter %%

\makeatletter
\def\operatorname#1{\mathop{\operator@font #1}\nolimits}%
\makeatother
\newcommand{\R}{\mathbb{R}}
\newcommand{\C}{\mathbb{C}}
\newcommand{\Q}{\mathbb{Q}}
\newcommand{\Z}{\mathbb{Z}}
\newcommand{\N}{\mathbb{N}}
\newcommand{\half}{{\frac{1}{2}}}

\newcommand{\Sp}{\operatorname{Sp}}

\newcommand{\A}{\mathcal{A}}

\newcommand{\abs}[1]{\left| #1 \right|}

\newcommand{\I} {\left[\, 0\, ,1\, \right]}

\newcommand{\mcz}{{\mu}_{\textrm{CZ}}}
\newcommand{\mmcz}{{\widehat{\mu}}_{\textrm{CZ}}}

\newcommand{\wh}{\widehat}

\newcommand{\p}{\partial}

\newcommand{\pf}{\longrightarrow}

\newcommand{\CZ}{{\rm CZ}}

\renewcommand{\AA}{\mathcal{A}}

\newcommand{\GG}{\mathcal{G}}
\newcommand{\EE}{\mathcal{E}}

\newcommand{\PP}{\mathcal{P}}

\DeclarePairedDelimiter\floor{\lfloor}{\rfloor}
\numberwithin{equation}{section}

\newenvironment{proof}[1][{}]{{ \textsc{Proof{#1}:~}}}{{\hfill$\square$\\ }}

\makeatletter
\newcommand*{\rom}[1]{\expandafter\@slowromancap\romannumeral #1@}
\makeatother

\def\adots{\mathinner{\mkern2mu\raise 1pt\hbox{.}\mkern 3mu\raise
4pt\hbox{.}\mkern2mu\raise 8pt\hbox{{.}}}}

\title{On the minimal number of periodic orbits on some hypersurfaces in $\R^{2n}$}
\author{Jean Gutt\footnote{This paper was written while JG was a postdoc at UC Berkeley under the supervision of Michael Hutchings.}\\
Department of Mathematics\\
University of Georgia\\
Athens, GA 30602\\
USA
\and
Jungsoo Kang\\
Mathematisches Institut\\
Westf\"alische Wilhelms-Universit\"at M\"unster\\
Einsteinstrasse 62, 48149 M\"unster\\
Germany
}

\date{ }                                           % Activate to display a given date or no date

\begin{document}
\maketitle

\begin{abstract}
	We study periodic orbits of the Reeb vector field on a nondegenerate dynamically convex starshaped hypersurface in $\R^{2n}$ along the lines of Long and Zhu \cite{LZ}, but using properties of the $S^1$- equivariant symplectic homology.
	We prove that there exist at least $n$ distinct simple periodic orbits on any nondegenerate starshaped hypersurface in $\R^{2n}$ satisfying the condition that the minimal Conley-Zehnder index is at least $n-1$.
	The condition is weaker than dynamical convexity.
\end{abstract}

%%%%%%%%%%%%%%%%%%%%%%%%%%%%%%%%%%%%%%%%%%%%%%%%%%%%%%%%%%%%%%%%%%%%%%%%%%%%%%%%%%%%%%%%%%%%%%%%%%%%%%%%%%%%%%%%%%%%%%%%%%%%%%%%%%%%%%%%%%%%%%%%
\section{Introduction}
We consider a starshaped hypersurface $\Sigma$ in $\R^{2n}$ endowed with the standard contact form $\alpha$ which is the restriction of the $1$-form $\lambda$ on $\R^{2n}$ defined by
\[
	\lambda=\half\sum_{j=1}^{n}(x^{j}dy^{j}-y^{j}dx^{j}).
\]
The \emph{Reeb vector field} $R_{\alpha}$ associated to a contact form $\alpha$ is the unique vector field on $\Sigma$ characterized by:
$\iota(R_{\alpha})d\alpha = 0$ and $\alpha (R_{\alpha}) = 1$.
Since this vector field does not vanish anywhere, there are no fixed points of its flow, and
periodic orbits are the most noticeable objects of its flow. 

The existence of a periodic orbit is known from Rabinowitz \cite{Rab79} and a long-standing question is to know the  (minimal) number of geometrically distinct periodic orbits of $(\Sigma,\alpha)$. 
This question has been studied in depth in the lowest dimensional case, in which the question is nontrivial, i.e. for a hypersurface $\Sigma\subset\R^4$ in \cite{HWZ2, HWZ3,HT, GH,GHHM, LL2,GGo}.
It turns out that, in this case, $(\Sigma,\alpha)$ carries at least two simple periodic orbits and if there are more than two simple periodic orbits, infinitely many of them are guaranteed generically.
In higher dimensions, nearly all known multiplicity results concern hypersurfaces in $\R^{2n}$ which satisfy some geometric conditions and appear in \cite{EL, blmr, EH87, LZ, WHL,Wang13}. 

This paper is based on the approach due to Long and Zhu \cite{LZ}.
They prove a certain lower bound on the number of simple periodic orbits on a strictly convex hypersurface.
In particular, they show that this lower bound equals  $n$ if the hypersurface is strictly convex and nondegenerate\footnote{A hypersurface is nondegenerate if all the periodic orbits are nondegenerate, i.e. 1 is not an eigenvalue of the linearized Poincar\'e return map; see Section \ref{index}}.
In their proof,  strict convexity plays a role twice.
First they use the fact that the index of periodic orbits behaves very well under iteration in the strictly convex case.
We show here that this remains true under the more general assumption of dynamical convexity.
Recall that $(\Sigma,\alpha)$ is {\em dynamically convex} if every periodic orbit has Conley-Zehnder index at least $n+1$; this is the case whenever $\Sigma$ is strictly convex.
Secondly, they use a result of \cite{EH87} to get information about the  interval where the indices of periodic orbits of $(\Sigma,\alpha)$ sit.
For this they use the Clarke dual action functional, which exists only when $\Sigma$ is strictly convex.
By using the positive $S^1$-equivariant symplectic homology instead, we observe that the idea of \cite{LZ} works under a weaker assumption and proves a stronger statement.
We now state the results proven in this paper.

A simple periodic orbit is called {\em even} if the Conley-Zehnder indices of all its iterates have the same parity; or, equivalently, if the linearized Poincar\'e return map has a number of real negative eigenvalues which is a multiple of $4$.
\begin{theorem}\label{thm of Long and Zhu}
	If a starshaped hypersurface $(\Sigma,\alpha)$ in $\R^{2n}$ is nondegenerate and dynamically convex, there are at least $n$ even simple periodic orbits.
	Moreover if there are precisely n simple periodic orbits, all periodic orbits have different indices.
\end{theorem}
This Theorem is proved in Section \ref{results} as Theorem \ref{thm:longmult}, with the dynamical convexity assumption slightly weakened.

A diffeomorphism $f:(\Sigma,\alpha)\to(\Sigma,\alpha)$ is called an {\em (anti-)strict contactomorphism} if $f^*\alpha=\alpha$ (if $f^*\alpha=-\alpha$).
Next corollary directly follows from the fact that an (anti-)strict contactomorphism maps a periodic orbit $\gamma$ to a periodic orbit $\gamma'=f\circ\gamma$ (respectively $\gamma'(t)=f\bigl(\gamma(T-t)\bigr)$) with the same period $T$ and the same Conley-Zehnder index. 
\begin{cor}\label{cor1.2}
	Suppose that a nondegenerate starshaped hypersurface $(\Sigma,\alpha)$ in $\R^{2n}$ is dynamically convex and  possesses precisely $n$ simple periodic orbits.
	If there is a (anti-)strict contactomorphism from $(\Sigma,\alpha)$ to itself, all periodic orbits are invariant under it. 
\end{cor}
An interesting class of (anti-) strict contactomorphisms arises when the hypersurface $\Sigma$ is invariant under a symmetry of $(\R^{2n}=\C^n,\lambda)$.
For example, let $f:\C^n\to \C^n$, $(z_1,\dots,z_n)\mapsto ( e^{2q_1\pi i}z_z,\dots,e^{2q_n\pi i}z_n)$, $q_1,\dots,q_n\in\N$ or $(z_1,\dots,z_n)\mapsto(\bar z_1,\dots,\bar z_n)$ and 
assume $\Sigma$ is invariant under $f$, i.e. $f(\Sigma)=\Sigma$.
Then the corollary yields that if there are precisely n periodic orbits, all of them are symmetric (i.e. invariant under the symmetry).
In low dimensional cases, this result is proved in \cite{Wang09,LLWZ} for a particular symmetry, but without the nondegeneracy assumption.

A nondegenerate contact form $\alpha$ is called {\em perfect} if the number of good periodic nondegenerate orbits  with Conley Zehnder index $k$ is equal to the dimension of the $k$-th  positive $S^1$-equivariant symplectic homology group. The following corollary generalizes a result due to G\"urel \cite{Gurel}.
We note from Theorem \ref{thm of Long and Zhu} that if $(\Sigma,\alpha)$ is dynamically convex and has precisely $n$ periodic orbits, it is perfect by degree reason (see Section \ref{results}).
\begin{cor}\label{cor:perfect contact form}
	Suppose that a nondegenerate contact form $\alpha$ on a starshaped hypersurface $\Sigma$ in $\R^{2n}$ is perfect.
	Then there are precisely n even simple periodic orbits.
\end{cor}
This is proved as Corollary \ref{perfect} in Section \ref{results}.

A natural question is whether dynamical convexity is necessary for multiplicity results.
The following Theorem (proven as Theorem \ref{mult} in Section \ref{results}) is our partial answer.
\begin{theorem}\label{thmintro:mult}
	Let $(\Sigma,\alpha)$ be a nondegenerate starshaped hypersurface in $\R^{2n}$ such that every periodic orbit has Conley-Zehnder index at least $n-1$.
	Then $\Sigma$ possesses at least $n$ simple periodic orbits.
\end{theorem}

We point out that every periodic geodesic flow of a Finsler $n$-sphere has at least Conley-Zehnder index $n-1$ under a certain pinching condition.
Under this pinching condition and nondegeneracy, Wang \cite[Theorem 1.2]{Wang} proved a conjecture of Anosov on the number of periodic geodesics on Finsler spheres.
The proof of Theorem \ref{thmintro:mult} can be used to give an alternative rather short proof of this result.
This will be discussed in a future paper.

It is easy to show that every nondegenerate starshaped hypersurface in $\R^{2n}$ has two periodic orbits, see for example \cite{K,Gurel}.
The following statement shows that if two periodic orbits on $(\Sigma,\alpha)$ do not satisfy a certain action-index resonance relation, there has to be a third one.
This can be thought of as a generalisation of a theorem due to Ekeland and Hofer \cite[Corollary 1]{EH87} (or see [Corollary V.3.17]\cite{E}). 

\begin{proposition}\label{thm:third periodic orbit}
	Let $(\Sigma,\alpha)$ be a nondegenerate starshaped hypersurface in $\R^{2n}$, for $n$ odd, with two simple periodic orbits $\gamma$ and $\delta$. Then $\Sigma$ carries another simple periodic orbit unless
	\begin{equation}\label{eq:action-index resonance relation}
		\frac{\A(\gamma)}{\mmcz(\gamma)}=\frac{\A(\delta)}{\mmcz(\delta)}
	\end{equation}
	where $\mmcz$ and $\AA$ stand for the mean Conley-Zehnder index and the action respectively.
\end{proposition}

The rest of the paper is divided into three sections.
Section \ref{index} is devoted to Long's index iteration formula and gives a proof of our slight generalisation of the common index jump Theorem due to Long and Zhu.
In Section \ref{SHS1}, we recall the properties of positive $S^1$-equivariant symplectic homology that we need.
Section \ref{results} contains the proofs of Theorem \ref{thm of Long and Zhu}, Corollary \ref{cor:perfect contact form}, and  Theorem \ref{thmintro:mult}.
Section \ref{third periodic orbit} is entirely devoted to the proof of Proposition \ref{thm:third periodic orbit}.

%---------------------------------------------------------------------------------------------------------------------------
\subsection*{Acknowledgements}
JG thanks Peter Albers for his kind and fruitful hospitality in M\"unster and acknowledges support from the BAEF.
JK is supported by DFG grant KA 4010/1-1.
%We are thankful to the anonymous referee for the many improvements suggested.

%%%%%%%%%%%%%%%%%%%%%%%%%%%%%%%%%%%%%%%%%%%%%%%%%%%%%%%%%%%%%%%%%%%%%%%%%%%%%%%%%%%%%%%%%%%%%%%%%%%%%%%%%%%%%%%%%%%%%%%%%%%%%%%%%%%%%%%%%%%%%%%%
\section{The main tools}

%---------------------------------------------------------------------------------------------------------------------------
\subsection{Index iterations}\label{index}
The Conley-Zehnder index  associates an integer to any continuous path $\psi$ defined on the  interval $[0,1]$ with values in the group $\Sp(\R^{2n-2})$ of $2(n-1)\times 2(n-1)$ symplectic matrices, starting from the identity and ending at a matrix which does  not admit $1$ as an eigenvalue. 
This index is used, for instance, in the definition of the grading of Floer homology theories.
If the path $\psi$ were a loop with values in the unitary group, one could define an integer by looking at the degree of the loop in the circle defined by the (complex) determinant -or an integer power of it.
One uses a continuous map $\rho$ from the symplectic group $\Sp(\R^{2n-2})$ into $S^1$ and an ``admissible''
extension of  $\psi$ to a path $\widetilde{\psi} : [0,2] \rightarrow \Sp(\R^{2n-2})$  in such a way that $\rho^2\circ \widetilde{\psi}:[0,2]\rightarrow S^1$ is a loop. 
The Conley-Zehnder index of $\psi$ is defined as the degree of this loop 
\begin{equation*}
   \mu_{\textrm{CZ}}(\psi) :=\deg (\rho^2\circ \widetilde{\psi}).
\end{equation*}
Let $\phi^t$ denotes the flow of the Reeb vector field $R_\alpha$ on a starshaped hypersurface $\Sigma$ in $\R^{2n}$ endowed with the standard contact form $\alpha$. The linearized flow $T\phi^t$ respects the splitting $T\Sigma=\R R_\alpha\oplus\ker\alpha$, we have $T\phi^t|_{\ker\alpha}:\ker\alpha\to\ker\alpha$.
Throughout the paper we assume that all the periodic  orbits (including all iterates) are nondegenerate; this means that $1$ is not an eigenvalue of the linearized Poincar\'e return map $T\phi^T|_{\ker\alpha}(\gamma(0))$ 
of a periodic orbit $\gamma:[0,T]\to(\Sigma,\alpha)$ with $\gamma(0)=\gamma(T)$ and $\dot \gamma(t)=R_\alpha(\gamma(t))$. The Conley-Zehnder index
of a periodic orbit $\gamma$ is defined by
\[
	\mcz(\gamma):=\mcz(\psi_\gamma)
\]
where $\psi_\gamma(t)\in\Sp(\R^{2n-2})$, $t\in[0,1]$ is the linearized flow $T\phi^{Tt}|_{\ker\alpha}(\gamma(0))$ expressed in a symplectic trivialization of $\ker\alpha$ along $\gamma$ extendable over a capping disk of $\gamma$.
For a complete presentation of the Conley-Zehnder index we refer to \cite{CZ, SalamonZehnder, Sal,Lon02, AD,gutt2}.
The {\emph{mean Conley-Zehnder index}} of a periodic  orbit $\gamma$ is defined to be
\[
	\mmcz(\gamma) := \lim_{k\to\infty}\frac{\mcz(\gamma^k)}{k}.
\]

To begin with, we recall Long's index iteration formula in the nondegenerate case, and immediate consequences of this formula which are used in the proofs of our results;
the proof of this  theorem can be found in \cite[Section 8.3]{Lon02} or in \cite[Theorem 3.2]{K}.

\begin{theorem} (\cite{Lon02})\label{Longsiteration} 
	Given a nondegenerate periodic orbit $\gamma$, so that all its iterates are nondegenerate, there exist an integer $p\in\Z$, an integer $q\in[\,0\,,n-1\,]$ and $q$ irrational numbers $\theta_{j}$ in $\I$, such that, for any positive integer $\ell \in\N$, the Conley-Zehnder index of the $\ell$-th iterate of $\gamma$ is given by
\begin{equation}\label{iteration}
	\mcz(\gamma^{\ell})=\ell p+2\sum_{j=1}^{q}\floor{\ell \theta_{j}} + q
\end{equation}
where $\floor{r}$ denotes the largest integer which is lower or equal to $r$, and where  $q$ can be $n-1$ only when $p$ is even. In particular,
\begin{equation}\label{eq:mean}
	\mcz(\gamma)=p+q,\quad \mmcz(\gamma)=p+2\sum_{j=1}^{q}\theta_{j}, 
\end{equation}
and 
\begin{equation}\label{eq:iteration and mean index}
	\big|\mu_\CZ(\gamma^\ell)-\ell\,\wh\mu_\CZ(\gamma)\big|< n-1.
\end{equation}
Moreover if $\mcz(\gamma)\ge n-1+c$ for some $c\in\N\cup\{0\}$, then $p\ge c$ and $\mcz(\gamma^{\ell +1})\ge\mcz(\gamma^{\ell})+c$.
The Conley-Zehnder indices of all even (resp. odd) iterates of a periodic orbit have the same parity.
\end{theorem}
An alternative way to see \eqref{eq:iteration and mean index} is presented in \cite[Lemma 3.4]{SalamonZehnder}.

The following theorem,  called the common index jump theorem due to Long and Zhu \cite[Theorem 4.3]{LZ}, is a key tool of the paper. We include a proof of the theorem stated below, because their idea in fact proves a slightly generalised statement which will be used  later in the paper. In the original proof, they used Bott's iteration formula and included the degenerate case; here we treat  the nondegenerate case which is simple enough for a proof only using Long's iteration formula.

\begin{theorem}(\cite{LZ})\label{thmCIJ+}
	Let $\gamma_1,\ldots,\gamma_k$ be  simple periodic  orbits on a given contact manifold of dimension $2n-1$.
	Assume that all the iterates of the periodic orbits are nondegenerate and that all the mean indices of the periodic orbits are positive; $\mmcz(\gamma_i)>0$ for all $i\in[\,0\,,k\,]$.
	Then, for any given $M\in\N$,  there exist infinitely many $N\in\N$ and $(m_1,\ldots,m_k)\in\N^{k}$ such that for any $m\in\{1,\ldots,M\}$
	\[
		\mcz\bigl(\gamma_i^{2m_i-m}\bigr)=2N-\mcz(\gamma_i^m)\quad\textrm{ and }\quad\mcz\bigl(\gamma_i^{2m_i+m}\bigr)=2N+\mcz(\gamma_i^m)
	\]
	and
	\[
		2N-(n-1)\leq\mcz(\gamma_i^{2m_i})\leq 2N+(n-1).
	\]
\end{theorem}
\begin{proof}
	Let $v$ be the vector in $\R^{k+\sum_{i=1}^{k}q_i}$ defined by
	\[
		v:=\biggl(\frac{1}{\mmcz(\gamma_1)},\ldots,\frac{1}{\mmcz(\gamma_k)},\frac{\theta_{1,1}}{\mmcz(\gamma_1)},\ldots,\frac{\theta_{1,q_1}}{\mmcz(\gamma_1)},\frac{\theta_{2,1}}{\mmcz(\gamma_2)},\ldots,
		\frac{\theta_{k,q_k}}{\mmcz(\gamma_k)}\biggr).
	\]
where $\mcz(\gamma_i^\ell)=\ell p_i+2\sum_{j=1}^{q_i}\floor{\ell \theta_{i,j}} + q_i$.
	Consider the closure of the projection on the torus $T^{k+\sum_{i=1}^{k}q_i}=\R^{k+\sum_{i=1}^{k}q_i}/\raisebox{-1ex}{$\Z^{k+\sum_{i=1}^{k}q_i}$}$ of  the set $\{k'v\}_{k'\in\N}$; it  is a closed subgroup of the torus $T^{k+\sum_{i=1}^{k}q_i}$. Hence  any neighbourhood of the neutral element of the torus contains the image of an infinite number of elements of the set $\{k'v\}_{k'\in\N}$.  
Hence, if we denote by $[a]$ the non integer part of $a$, $[a]:=a-\floor{a}$, for any given $\epsilon>0$,  there exist infinitely many $N\in\N$ such that all
	\[
		\bigg[\frac{N\theta_{i,j}}{\mmcz(\gamma_i)}\biggr] \textrm{ and } \bigg[\frac{N}{\mmcz(\gamma_i)}\biggr] \textrm{ are in } [\,0\,,\epsilon\,[   \textrm{ or in  }  ]\,1-\epsilon\, , 1\,[.
	\]
With $N$ as above, if  $\Big[\frac{N}{\mmcz(\gamma_i)}\Big] \textrm{ is in } [\,0\,,\epsilon\,[ $ define  $m_i := \floor*{\frac{N}{\mmcz(\gamma_i)}}$ and $\eta_i=1$. Then
\[
		[2m_i\theta_{i,j}] = \Biggl[\floor*{\frac{N}{\mmcz(\gamma_i)}}2\theta_{i,j}\Biggr] = \Biggl[\frac{2N\theta_{i,j}}{\mmcz(\gamma_i)}-\bigg[\frac{N}{\mmcz(\gamma_i)}\biggr]2\theta_{i,j}\Biggr]
\]
	 lies in $[\,0\,,4\epsilon\,[\,\,\cup\,\,]\,1-4\epsilon\,,1\,[$.  If  $\bigg[\frac{N}{\mmcz(\gamma_i)}\biggr] \textrm{ is in } ]\,1-\epsilon\, , 1\,[ $, define  $m_i := -\floor*{\frac{-N}{\mmcz(\gamma_i)}}$ and $\eta_i=-1$ . Then
\[
		[2m_i\theta_{i,j}] = \Biggl[-\floor*{\frac{-N}{\mmcz(\gamma_i)}}2\theta_{i,j}\Biggr] = \Biggl[\frac{2N\theta_{i,j}}{\mmcz(\gamma_i)}+\bigg[\frac{-N}{\mmcz(\gamma_i)}\biggr]2\theta_{i,j}\Biggr]
\]
	 lies in $[\,0\,,4\epsilon\,[\,\,\cup\,\,]\,1-4\epsilon\,,1\,[$. Observe that $\Big[\frac{-N}{\mmcz(\gamma_i)}\Big] \textrm{ is in } [\,0\,,\epsilon\,[ $.
	 Hence, with our definitions, we always have
\begin{equation}\label{eq:defmietai}
	 \Big[\frac{\eta_i N}{\mmcz(\gamma_i)}\Big] \in  [\,0\,,\epsilon\,[ , \qquad m_i := \eta_i\floor*{\frac{\eta_i N}{\mmcz(\gamma_i)}}, \quad  \textrm { and }  \, [2m_i\theta_{i,j}]\in [\,0\,,4\epsilon\,[\,\,\cup\,\,]\,1-4\epsilon\,,1\,[ .
\end{equation}	 
	 For each $i\in\{1,\ldots,m\}$, we denote by $\mathcal{E}_i$ the set
	\[
		\mathcal{E}_i := \bigl\{j\in\{1,\ldots,q_i\}\,\big|\, [2m_i\theta_{i,j}] \in [\,0\,,4\epsilon\,[ \bigr\}
	\]
	and by $\mathcal{E}_i^c$ its complementary ($\mathcal{E}_i^c := \{1,\ldots,q_i\}\setminus \mathcal{E}_i$).
	
	Given a positive  integer $M$ we pick the $\epsilon$ such that
	$$
	4 \epsilon <  \min\bigl\{ \theta_{i,j}\,, \left[2\theta_{i,j}\right],\ldots, \left[M\theta_{i,j}\right]\, ,1-\theta_{i,j}\,, \left[1-2\theta_{i,j}\right]\,,\ldots, \left[1-M\theta_{i,j})\right]\,,\tfrac{1}{6q_i},\tfrac{1}{\mmcz(\gamma_i)}\,\bigr|\, \forall i,j \bigr\}.
	$$
	For any $N$ corresponding as above to this $\epsilon$ and with the corresponding $m_i$, we have, 
	\[ 
	\bigl[2m_i\theta_{i,j}\bigr] < 4\epsilon\quad \forall j\in\mathcal{E}_i \qquad \textrm{ and }\qquad  1- \bigl[2m_i\theta_{i,j}\bigr] <4\epsilon\quad \forall j\in\mathcal{E}_i^c.
	\]
	Thus we have $\bigl[2m_i\theta_{i,j}\bigr]-\theta_{i,j}<4\epsilon-\theta_{i,j}<0$ and $\bigl[2m_i\theta_{i,j}\bigr]+\theta_{i,j}<1$ for all $ j\in\mathcal{E}_i$, and 
	$\bigl[2m_i\theta_{i,j}\bigr]-\theta_{i,j}>1-4\epsilon-\theta_{i,j}>0$ and $\bigl[2m_i\theta_{i,j}\bigr]+\theta_{i,j}>1-4\epsilon+\theta_{i,j}>1$ for all $ j\in\mathcal{E}^c_i$, so that
	\begin{equation}\label{eq:mitetai}
	\begin{array}{llll} 
	      \floor*{2m_i\theta_{i,j}}=\floor*{(2m_i-1)\theta_{i,j}}  &\textrm{ and  }&[2m_i\theta_{i,j}]=[(2m_i-1)\theta_{i,j}]+\theta_{i,j} &\textrm{ for } j\in\EE_i^c,\\[1ex]
		\floor*{2m_i\theta_{i,j}}=\floor*{(2m_i-1)\theta_{i,j}}+1&\textrm{ and  }&[2m_i\theta_{i,j}]=[(2m_i-1)\theta_{i,j}]+\theta_{i,j} -1   &\textrm{ for } j\in\EE_i,\\[1ex]
		\floor*{2m_i\theta_{i,j}}=\floor*{(2m_i+1)\theta_{i,j}} &\textrm{ and  }& [(2m_i+1)\theta_{i,j}]=[2m_i\theta_{i,j}]+\theta_{i,j}   &\textrm{ for } j\in\EE_i,\\[1ex]
		\floor*{2m_i\theta_{i,j}}=\floor*{(2m_i+1)\theta_{i,j}}-1 &\textrm{ and  }& [(2m_i+1)\theta_{i,j}]=[2m_i\theta_{i,j}]+\theta_{i,j}-1   &\textrm{ for } j\in\EE_i^c.
	\end{array}
\end{equation}

Equation \eqref{eq:mean} reads $\mmcz(\gamma_i)=p_i +2\sum_{j=1}^{q_i} \theta_{i,j}$ and yields:
	\begin{align*}
		2m_ip_i + 2\sum_{j=1}^{q_i}\floor*{2m_i\theta_{i,j}} &=2m_i\mmcz(\gamma_i)+2\sum_{j=1}^{q_i}\bigl(\floor*{2m_i\theta_{i,j}}-2m_i\theta_{i,j}\bigr)\\
		&=\eta_i \floor*{\frac{\eta_i N}{\mmcz(\gamma_i)}}2\mmcz(\gamma_i)-2\sum_{j=1}^{q_i}\bigl[2m_i\theta_{i,j}\bigr]\\
		&= 2N- \eta_i\biggl[\frac{\eta_i N}{\mmcz(\gamma_i)}\biggr]2\mmcz(\gamma_i)-2\sum_{j=1}^{q_i}\bigl[2m_i\theta_{i,j}\bigr];
	\end{align*}
with our choices of $\epsilon$, $N$  $m_i$'s and $\eta_i$'s, using \eqref{eq:defmietai} and the fact that the $\mmcz(\gamma_i)$'s are positive, we have
	\begin{align*}
		\abs{2m_ip_i+2\sum_{j=1}^{q_i}\floor*{2m_i\theta_{i,j}}-2N+2\#\mathcal{E}_i^c} &\leq 2\sum_{j\in\mathcal{E}_i^c}\Bigl(1-\bigl[2m_i\theta_{i,j}\bigr]\Bigr)\\
			&\quad+2\sum_{j\in\mathcal{E}_i}\Bigl(\bigl[2m_i\theta_{i,j}\bigr]\Bigr) + \biggl[\frac{\eta_i N}{\mmcz(\gamma_i)}\biggr]2\mmcz(\gamma_i)\\
		& <8\epsilon\#\mathcal{E}_i^c +8\epsilon\#\mathcal{E}_i+ 2 \epsilon \mmcz(\gamma_i)= 8q^i \epsilon +2\mmcz(\gamma_i){\epsilon}<1.
	\end{align*}
Since the difference of two integers is still an integer, this in turn implies 
	\begin{equation}\label{eq:NmathcalE}
		2m_ip_i+2\sum_{j=1}^{q_i}\floor*{2m_i\theta_{i,j}} = 2N-2\#\mathcal{E}_i^c.
	\end{equation}
	Equation \eqref{iteration} gives $\mcz(\gamma_i^{2m_i})=2m_ip_i+2\sum_{j=1}^{q_i} \floor{2m_i\theta_{i,j}}+q_i$; hence 
	\[
		\mcz(\gamma_i^{2m_i})=2N-2\#\mathcal{E}_i^c+q_i\in [\,2N-(n-1)\,,\,2N+(n-1)\,]
	\]
	and this proves the last part of the statement.
	We now  compute $\mcz(\gamma_i^{2m_i\pm1})$, using  equation \eqref{iteration} and relations \eqref{eq:mitetai} and \eqref{eq:NmathcalE} :
	\begin{align*}
		\mcz(\gamma_i^{2m_i-1}) &=2m_ip_i -p_i+2\sum_{j=1}^{q_i} \floor{(2m_i-1)\theta_{i,j}}+q_i\\
		&=2N+2\sum_{j=1}^{q_i}\Bigl(\floor*{(2m_i-1)\theta_{i,j}}-\floor*{2m_i\theta_{i,j}}\Bigr)-p_i+q_i-2\#\mathcal{E}_i^c\\
		&=2N+2\sum_{j\in\mathcal{E}_i}(-1)-p_i+q_i-2\#\mathcal{E}_i^c
		=2N-p_i-q_i =2N-\mcz(\gamma_i)
	\end{align*}
and
	\begin{align*}
		\mcz(\gamma_i^{2m_i+1}) &=2m_ip_i +p_i+2\sum_{j=1}^{q_i} \floor{(2m_i+1)\theta_{i,j}}+q_i\\
		&=2N+2\sum_{j=1}^{q_i}\Bigl(\floor*{(2m_i+1)\theta_{i,j}}-\floor*{2m_i\theta_{i,j}}\Bigr)+p_i+q_i-2\#\mathcal{E}_i^c\\
		&=2N+2\sum_{j\in\mathcal{E}_i^c}1+p_i+q_i-2\#\mathcal{E}_i^c = 2N+\mcz(\gamma_i).
	\end{align*}
More generally,	for any positive integer $1\le m\le M$, we have 
	\begin{align*}
		\mcz(\gamma_i^{2m_i+m}) &=2m_ip_i +mp_i+2\sum_{j=1}^{q_i} \floor{(2m_i+m)\theta_{i,j}}+q_i\\
		&= 2N+2\sum_{j=1}^{q_i}\Bigl(\floor*{(2m_i+m)\theta_{i,j}} - \floor*{2m_i\theta_{i,j}}\Bigr)+mp_i+q_i-2\#\mathcal{E}_i^c\\
		&=2N+mp_i+2\sum_{j=1}^{q_i}\Bigl(\floor*{(2m_i+m)\theta_{i,j}} - \floor*{(2m_i+1)\theta_{i,j}}\Bigr)+q_i\\
		&=2N+mp_i+2\sum_{j=1}^{q_i}\Bigl(\floor*{\bigl[(2m_i+1)\theta_{i,j}\bigr]+(m-1)\theta_{i,j}}\Bigr)+q_i\\
		&=2N+mp_i+2\sum_{j=1}^{q_i}\floor*{m\theta_{i,j}}+q_i=2N+\mcz(\gamma_i^m)
	\end{align*} 
In the fourth equality we used the identity
$$
\floor*{a+b}=\floor*{a}+\floor*{[a]+b},\quad  \forall a,b\in\R \textrm{ hence }  \floor*{a+b}-\floor*{a}=\floor*{[a]+b}
$$
 { for}  $a=(2m_i+1)\theta_{i,j}$ and  $ b=(m-1)\theta_{i,j}$ .
For the last equality we compute that if $j\in\EE_i$, 
$$
\floor*{\bigl[(2m_i+1)\theta_{i,j}\bigr]+(m-1)\theta_{i,j}}=\floor*{\bigl[ 2m_i\theta_{i,j}\bigr]+m\theta_{i,j}}
=\floor*{m\theta_{i,j}}
$$
using \eqref{eq:mitetai} and the fact that $[m\theta_{i,j}]+[2m_i\theta_{i,j}]<[m\theta_{i,j}]+4\epsilon<1$. If $j\in\EE_i^c$,
$$
\floor*{\bigl[(2m_i+1)\theta_{i,j}\bigr]+(m-1)\theta_{i,j}}=\floor*{m\theta_{i,j}-1+[2m_i\theta_{i,j}]}=\floor*{m\theta_{i,j}}
$$
using again \eqref{eq:mitetai} and the fact that $1-[2m_i\theta_{i,j}]<4\epsilon<1-[m\theta_{i,j}]$.
The computation for $\mcz(\gamma_i^{2m_i-m})$ is analogous.
\end{proof}

%---------------------------------------------------------------------------------------------------------------------------
\subsection{Positive $S^1$-equivariant symplectic homology}\label{SHS1}
Symplectic homology is defined for a compact symplectic manifold with nondegenerate contact type boundary.
Very roughly, it is the semi-infinite dimensional Morse homology for the symplectic action functional defined on the contractible component of the free loop space of such symplectic manifolds.
In our situation, a nondegenerate starshaped hypersurface $\Sigma$ in $\R^{2n}$ is a contact type boundary of the compact region bounded by $\Sigma$.
The version of homology which we will use is the so called  positive $S^1$-equivariant symplectic homology for $(\Sigma,\alpha)\subset\R^{2n}$ with rational coefficients, denoted by $SH^{S^1,+}_*(\Sigma,\R^{2n};\Q)$.
The $S^1$-action we are referring to is the reparametrization action on the free loop space and by positive we mean that only periodic orbits of the Reeb vector field are taken into account.
Rather than giving a precise definition we recall some important properties of it.
For details we refer the reader to \cite{V,Seidel,BOjems,bo,BOind,gutt}.
We can think that the chain complex for $SH^{S^1,+}_*(\Sigma,\R^{2n};\Q)$ is built over {\em unparametrized} periodic orbits of $(\Sigma,\alpha)$ with  grading  given by the Conley-Zehnder index, in light of \cite{bo}, see also \cite{K,gutt}.
The differential is counting gradient flow trajectories of the action functional between periodic orbits modulo the $S^1$-action, which solve a certain elliptic PDE.
Moreover bad periodic orbits do not contribute to this  homology.
Recall that a periodic orbit $\gamma$ is called {\em good} if the parity of its Conley-Zehnder index  is the same as that of the underlying simple orbit and is called {\em bad} otherwise.
 
More precisely, for any large real number $K$, there exists an integer $N$, such that the $S^1$-equivariant symplectic homology $SH^{S^1,+}_*(\Sigma,\R^{2n};\Q)$, truncated at level $K$ for the action, and up to degree $N$, is the  limit of homologies which can be computed via a spectral sequence for which the complex of the first page up to degree $N$ is spanned by the good periodic  orbits  on the boundary $\partial \Sigma$ of period  at most $K$,  graded  by their Conley-Zehnder index, and with a differential $\partial$, so that the action $\mathcal{A}(\gamma):=\int_\gamma\alpha$ of a periodic orbit decreases along $\partial$ (see \cite{gutt}).

The following computation is by now well known.
\begin{theorem}\label{thm:computing}
	
	Let $\Sigma$ be a nondegenerate starshaped hypersurface in $\R^{2n}$. Then we have
\[
		SH_*^{S^1,+}(\Sigma,\R^{2n};\Q) \cong 
	\begin{cases}
		\,\Q &\textrm{if } *\in n-1+2\N_{\geq 1}\\[1ex]
		\,0 &\textrm{otherwise}.
	\end{cases}
\]

\end{theorem}
It implies in particular that for each non negative integer $m$ there exists at least one good periodic orbit of Conley-Zehnder index $n+1+2m$.
It also implies that if there exists a good periodic orbit with Conley-Zehnder index equal to $n+2m$, then there must exist at least $1$ extra good periodic orbit of order $n+2m+1$ or $n+2m-1$. Remark that the hypersurface is perfect if and only if  for each integer $m\ge 0$ there is exactly one good periodic orbit with Conley Zehnder index $n+1+2m$ and there are no good periodic orbit of any other Conley Zehnder index.
%{\color{red}
%It is worth mentioning that a resonance relation between mean Conley-Zehnder indices of them follows from the above %computation. We assign to each simple periodic orbit $\gamma$ a constant $c_\gamma\in\{1,2\}$ in the following way. %If $\gamma^2$ is good, $c_\gamma=1$ and otherwise, $c_\gamma=2$.  If there are only finitely many simple periodic %orbits, the resonance relation reads
%\[
%	\sum\frac{(-1)^{\mu_\CZ(\gamma)}}{c_\gamma\,\wh\mu_\CZ(\gamma)}=\frac{1}{2}
%\]
%where the sum runs over all simple periodic orbits. This is often helpful to find multiple periodic orbits. 
%We refer to \cite{Vit89,GK,WHL,GGo} for references.}

%%%%%%%%%%%%%%%%%%%%%%%%%%%%%%%%%%%%%%%%%%%%%%%%%%%%%%%%%%%%%%%%%%%%%%%%%%%%%%%%%%%%%%%%%%%%%%%%%%%%%%%%%%%%%%%%%%%%%%%%%%%%%%%%%%%%%%%%%%%%%%%%
\section{Multiplicity of periodic orbits}\label{results}
We denote by $\mathcal{P}_{n+1}$ the set of periodic orbits on $(\Sigma,\alpha)$ whose Conley-Zehnder indices are congruent to $n+1$ modulo 2.

\begin{theorem}\label{thm:longmult}
	Let $(\Sigma,\alpha)$ be a nondegenerate starshaped hypersurface in $\R^{2n}$. Suppose that every simple periodic orbit in $\mathcal{P}_{n+1}$ has Conley-Zehnder index at least $n+1$.
	Then $(\Sigma,\alpha)$ possesses at least $n$ simple periodic orbits, all iterations of which are in $\mathcal{P}_{n+1}$.
\end{theorem}
\begin{proof}
	Knowing the positive $S^1$-equivariant symplectic homology from Theorem  \ref{thm:computing}, which has generators in all degrees which are congruent to $n+1$ modulo $2$, no iterate of a simple periodic orbit not in $\mathcal{P}_{n+1}$ can generate a nonzero homology class since if some iterate of this is in $\mathcal{P}_{n+1}$, it is a bad periodic orbit.
	We can assume without loss of generality that there are only a finite number of simple periodic orbits in $\mathcal{P}_{n+1}$, say $\gamma_1,\ldots,\gamma_k$. Periodic orbits with Conley-Zehnder indices at least 
	$n+1$ have positive mean indices (cf equation \eqref{eq:iteration and mean index}) and thus
	by Theorem \ref{thmCIJ+}, with $M=1$, there exists an interval 
	$$]\,2N-(n+1)\,,2N+(n+1)\,[$$ for some $N\in\N$,  in which the Conley-Zehnder index of precisely one iterate of each of those orbits sits.
	Indeed, we have, with the notations of that theorem, $$\mcz\bigl(\gamma_i^{2m_i-1}\bigr)=2N-\mcz(\gamma_i)\le2N-(n+1),$$
	$$\mcz\bigl(\gamma_i^{2m_i+1}\bigr)=2N+\mcz(\gamma_i)\ge 2N+(n+1)$$ and, by  Long's iteration formula (Theorem \ref{Longsiteration})
	$\mcz\bigl(\gamma_i^{k}\bigr)<\mcz\bigl(\gamma_i^{k+1}\bigr)$ for all $k\in\N$.
	In view of Theorem \ref{thm:computing} again, there must be generators in the $n$ degrees which correspond to the Conley-Zehnder indices in the interval  (i.e. indices
	$2N-(n-1), 2N-(n-3), \ldots, 2N+(n-3), 2N+n-1$). Since they can only correspond to $\gamma_1^{2m_1},\ldots, \gamma_k^{2m_k}$, all of them have to be good and $k\geq n$.
\end{proof}

This together with the following corollary prove Theorem \ref{thm of Long and Zhu}.

\begin{cor}\label{cor:cormult}
	If a nondegenerate dynamically convex starshaped hypersurface $(\Sigma,\alpha)$ in $\R^{2n}$ possesses precisely $n$ simple periodic  orbits, then all periodic orbits are in $\mathcal{P}_{n+1}$ and all Conley-Zehnder indices of periodic orbits are different.
\end{cor}
\begin{proof}
	The first assertion directly follows from the theorem. 
	If two periodic orbits have the same index $n-1 +2k$ with $k\in\N$, there would exist a good periodic orbit with index $n+2k$ or $n+2(k+1)$ by Theorem \ref{thm:computing} and this is not in $\mathcal{P}_{n+1}$.
\end{proof}

\begin{cor}\label{perfect}
	Suppose that a nondegenerate contact form $\alpha$ on a starshaped hypersurface $\Sigma$ in $\R^{2n}$ is perfect.
	Then there are precisely n even simple periodic orbits.
\end{cor}
\begin{proof}
	From Theorem \ref{thm:longmult}, we know that there are at least $n$ even simple periodic orbits since perfectness implies dynamical convexity. Indeed if there is a periodic orbit whose Conley-Zehnder index is less than $n+1$, perfectness is violated since $SH_{n+1}^{S^1,+}(\Sigma,\R^{2n};\Q)$ is the first nonzero homology group, see Theorem \ref{thm:computing}. Now we show that there are at most $n$ even simple periodic orbits, see also \cite[Corollary 1.6]{Gurel}.
	Assume by contradiction that there are more than $n$ even simple periodic orbits.
	We choose $n+1$ even simple periodic orbits and then apply Theorem \ref{thmCIJ+}.
	Then there are $n+1$ good periodic orbits with index sitting in $[\,2N-(n-1)\,,\, 2N+(n-1)\,]$.
	By Theorem \ref{thm:computing}, this contradicts the perfectness assumption.
\end{proof}

This proves Corollary \ref{cor:perfect contact form}. Next we provide a proof of Theorem \ref{thmintro:mult}.

\begin{theorem}\label{mult}
	Let $(\Sigma,\alpha)$ be a nondegenerate starshaped hypersurface in $\R^{2n}$ such that all periodic orbits have Conley-Zehnder index at least $n-1$.
	Then $(\Sigma,\alpha)$ possesses at least $n$ simple periodic orbits.
\end{theorem}
\begin{proof}
       We study the complex built with the good periodic orbits and see its compatibility with the positive $S^1$-equivariant symplectic homology computation in Theorem \ref{thm:computing}.
	Due to Theorem \ref{thm:longmult}, we may assume that there is a periodic orbit $\Gamma$ whose Conley-Zehnder index is $n-1$.
	Using the same argument as in Theorem \ref{thm:longmult},
	we know that there exist at least  $n-2$ geometrically distinct simple periodic orbits $\gamma_1,\ldots,\gamma_{n-2}$   for which all  iterates are in $\mathcal{P}_{n+1}$ :	 we assume  that the  only  simple periodic orbits in $\mathcal{P}_{n+1}$ are $\gamma_1,\ldots,\gamma_k$. Since periodic orbits with Conley-Zehnder indices at least 
	$n-1$ have positive mean indices (cf equation \eqref{eq:iteration and mean index}),
	by Theorem \ref{thmCIJ+}, with $M=1$, there exists an interval 
	$$]\,2N-n+1\,,2N+n-1\,[$$ for some $N\in\N$,  in which the Conley-Zehnder index of precisely one iterate of each of those orbits sits.
	Indeed, we have, with the notations of that theorem, $$\mcz\bigl(\gamma_i^{2m_i-1}\bigr)=2N-\mcz(\gamma_i)\le2N-n+1,$$
	$$\mcz\bigl(\gamma_i^{2m_i+1}\bigr)=2N+\mcz(\gamma_i)\ge 2N+n-1$$ and, by  Long's iteration formula (Theorem \ref{Longsiteration})
	$\mcz\bigl(\gamma_i^{k}\bigr)\le\mcz\bigl(\gamma_i^{k+1}\bigr)$ for all $k\in\N$.
	In view of Theorem \ref{thm:computing} again, there must be generators in the $n-2$ degrees which correspond to the Conley-Zehnder indices in the interval  (i.e. indices
	$2N-n+3, 2N-n+5, \ldots, 2N+n-5, 2N+n-3$). They can only correspond to $\gamma_1^{2m_1},\ldots, \gamma_k^{2m_k}$, all of them have to be good so $k\geq n-2$.

	We also know that $SH^{S^1,+}_{n-1}=0$; since we have a generator $\Gamma$ in the chain complex in that degree,   there must exist a good periodic orbit $\delta$ of index $n$.
	Observe that $\delta$ cannot be an iterate of $\Gamma$ or any of the $\gamma_i$'s  because of the parity of its index (it would be a bad orbit). 
		This shows that either we already have $n$ simple periodic orbits and there is nothing more to prove, or $\Gamma$ is one of the orbits $\gamma_i$'s, say $\Gamma= \gamma_1$. We  assume by contradiction that $\gamma_1,\ldots,\gamma_{n-2},\delta$ are the only simple periodic orbits. 
	We can also assume that $\gamma_1$ is the only periodic orbit of Conley-Zehnder index $n-1$.
	Indeed  another periodic orbit of index $n-1$  would imply the existence of a second orbit $\widetilde{\delta}$ of index $n$; it would be 
	 geometrically distinct from $\delta$ since $\mcz(\delta^{m+1})\ge\mcz(\delta^m)+1$ and we would have shown the existence of $n$ simple periodic orbits.

	Thus we can assume that $\mcz(\gamma_1^2)\ge n+1$ and $\mcz(\gamma_i)\ge n+1$ for all $i\in\{2,\ldots,n-2\}$. Hence by Theorem \ref{Longsiteration},
$$
\mcz(\delta^s)>\mcz(\delta),\quad \mcz(\gamma_i^s)>\mcz(\gamma_i)
$$
for all $i\in\{2,\dots,n-2\}$ and for all integers $s\geq 2$. Since all $\gamma_1,\ldots,\gamma_{n-2},\delta$ have positive mean Conley-Zehnder indices, we can apply Theorem \ref{thmCIJ+}.   Let $(N,m_1,\ldots,m_{n-2},m_{\delta})\in\N^n$ be given by Theorem \ref{thmCIJ+} for $M=2$.
		We have, for all integers $s\ge 2$ and for all $i\in\{2,\ldots,n-2\}$:
	\[
		\mcz(\gamma_1^{2m_1-s})\le \mcz(\gamma_1^{2m_1-2})=2N-\mcz(\gamma_1^2)\le 2N-n-1 < \mcz(\gamma_1^{2m_1-1})=2N-n+1
		\]
		\[\quad\textrm{and}\quad 2N+n-1=\mcz(\gamma_1^{2m_1+1}) < \mcz(\gamma_1^{2m_1+2})\le\mcz(\gamma_1^{2m_1+s}).
	\]
	\[
		\mcz(\gamma_i^{2m_i-s}) < \mcz(\gamma_i^{2m_i-1})\le 2N-n-1\quad\textrm{and}\quad 2N+n+1\le \mcz(\gamma_i^{2m_i+1}) < \mcz(\gamma_i^{2m_i+s}).
	\]
	\[
		\mcz(\delta^{2m_\delta-s}) < \mcz(\delta^{2m_\delta-1})= 2N-n\quad\textrm{and}\quad 2N+n=\mcz(\delta^{2m_\delta+1}) < \mcz(\delta^{2m_\delta+s}).
	\]
	Hence the only  periodic orbits whose Conley-Zehnder indices lie in $[\,2N-n\,,2N+n\,]$ are $\delta^{2m_\delta-1}$ with index $2N-n$, $\gamma_1^{2m_1-1}$ with index $2N-n+1$,
	the $n-1$ orbits $\gamma_i^{2m_i}, 1\le i\le n-1$, whose indices are in $]\,2N-n+1\,,2N+n-1\,[\cap\{n+1+2\N\}$, $\delta^{2m_\delta+1}$ with index $2N+n$, $\gamma_1^{2m_1+1}$ with index $2N+n-1$, and $\delta^{2m_{\delta}}$ with index in $[2N-(n-1),2N+(n-1)]$.
	We distinguish two cases, whether $\delta^{2m_{\delta}}$ is good or bad.\\[1ex]
	{\bf{Case 1 :}} The even iterates of $\delta$ are good.
	Then the index  of $\delta^{2m_{\delta}}$ sits in $[\,2N-n+2\,,2N+n-2\,]$ and the orbit generates a $1$-dimensional piece in the complex and also in the homology since $\gamma_1^{2m_1-1},\gamma_1^{2m_1},\dots\gamma_{n-2}^{2m_{n-2}},\gamma_1^{2m_1+1}$ have to generate all homology classes of $SH^{S^1,+}$ with degrees in $[2N-n+1,2N+n-1]$ and therefore $\delta^{2m_{\delta}}$ is a cycle and not a boundary.
	This contradicts the computation of $SH^{S^1,+}$ given in  Theorem \ref{thm:computing}.\\[1ex]
	{\bf{Case 2 :}} The even iterates of $\delta$ are bad.
	We claim that $\mcz(\delta^3)\geq n+3$.
	Indeed by Theorem \ref{Longsiteration}, $\mcz(\delta^3)\geq n+2$ and  $\mcz(\delta^3)\neq n+2$ since otherwise $\mcz(\delta)= p+q$ with $p=1$ and $q=n-1$ which contradicts the fact that $p$ must be even if $q=n-1$ (cfr Theorem \ref{index}).
	This implies in particular that there are no periodic orbits of index $n+2$, therefore there is only one periodic orbit of index $n+1$.
	By Theorem \ref{thmCIJ+}, we know that
	\[
		\#\big\{\mcz^{-1}(2N+n+1\big)\}=1 \quad\textrm{and}\quad\#\big\{\mcz^{-1}(2N+n-1)\big\}=1
	\]
	and they generate the nonzero homology classes of $SH^{S^1,+}$ in each degree.
	But $\delta^{2m_\delta+1}$ is a good orbit of index $2N+n$, thus generates a homology class which is a contradiction with the computation of $SH^{S^1,+}$ given in Theorem \ref{thm:computing}.
\end{proof}

%%%%%%%%%%%%%%%%%%%%%%%%%%%%%%%%%%%%%%%%%%%%%%%%%%%%%%%%%%%%%%%%%%%%%%%%%%%%%%%%%%%%%%%%%%%%%%%%%%%%%%%%%%%%%%%%%%%%%%%%%%%%%%%%%%%%%%%%%%%%%%%%
\section{Third periodic orbit}\label{third periodic orbit}

This section is devoted to the proof of Proposition \ref{thm:third periodic orbit}.
Let $(\Sigma,\alpha)$ be a nondegenerate starshaped hypersurface in $\R^{2n}$.
In the case $n=2$, if there are precisely two periodic orbits, it is known that there is the action-index resonance relation between them, i.e. \eqref{eq:action-index resonance relation} holds, see \cite{BCE,Gurel}.
Now we consider the cases when $n\geq 3$.
From Theorem \ref{thm:computing}, we need at least one simple periodic orbit $\gamma\in\PP_{n+1}$ to generate non-zero homology classes.
Theorem \ref{thmintro:mult} shows that  there are at least n simple periodic orbits if all their Conley Zehnder indices are at least $n-1$, so  we may assume that $\mcz(\gamma)\leq n-3$.
Theorem \ref{thm:computing} says that the cohomoly vanishes in any degree $\leq n-3$, so we know that there is another simple periodic orbit $\delta$ such that $\delta^\ell$ is good with $\mcz(\delta^\ell)\in\{\mcz(\gamma)-1,\mcz(\gamma)+1\}$ for some $\ell\in\N$. Note that if some  iterates of $\delta$ are in $P_{n+1}$, they are bad.
We assume for a contradiction that $\gamma$ and $\delta$ are the only simple periodic orbits.
Note that both periodic orbits have positive mean indices since otherwise we need an additional periodic orbit to meet the homology computation in Theorem \ref{thm:computing}, in view of \eqref{eq:iteration and mean index}. Indeed, if the mean index is not positive, the indices of all iterates are less than $n-1$.
Denoting as before by $\N$ the set of strictly positive integers, we also may assume that 
\begin{equation}\label{eq:iterations_of_gamma}
	\{\mcz(\gamma^k)\,|\,k\in\N\}= \min \{\mcz(\gamma^k)\,|\,k\in\N\}-2+2\N
\end{equation}
since otherwise, by   Theorem \ref{thmCIJ+}, there is an infinite number of $q$'s in  the set $n-1+2\N$ which do not belong to $\{\mcz(\gamma^k)\,|\,k\in\N\}$ and this immediately guarantees an additional periodic orbit.

%%%%%%%%%%%%%%%%%%%%%%%
\subsection{First case: $\frac{\A(\gamma)}{\mmcz(\gamma)}>\frac{\A(\delta)}{\mmcz(\delta)}$}

Since bad periodic orbits do not have any contribution to the homology $SH^{S^1,+}$, we  consider in this section the Conley-Zehnder index only defined on the set $\GG$ of good periodic orbits:
\[
	\mu_\CZ:\GG\to\Z.
\]
Observe from \eqref{eq:iteration and mean index} that $\mu_\CZ(\gamma^k)=r$ implies $k\mmcz(\gamma)\in]\,r-(n-1)\,,r+(n-1)\,[$ and thus
\[
	(r+n-1)\frac{\AA(\gamma)}{\mmcz(\gamma)} > \AA(\gamma^k)=k\AA(\gamma) > (r-n+1)\frac{\AA(\gamma)}{\mmcz(\gamma)}.
\]
Similarly $\mu_\CZ(\delta^\ell)=r\pm1$ implies $\ell\mmcz(\delta)\in]\,r-1-(n-1)\,,r+1+(n-1)\,[$ and 
\[
	(r-n)\frac{\AA(\delta)}{\mmcz(\delta)} < \AA(\delta^\ell)=\ell\AA(\delta) < (r+n)\frac{\AA(\delta)}{\mmcz(\delta)}.	
\]
Hence
\[
\AA(\gamma^k)>\AA(\delta^\ell)\quad\textrm{ when } \quad \frac{r-n+1}{r+n}\frac{\AA(\gamma)}{\mmcz(\gamma)} > \frac{\AA(\delta)}{\mmcz(\delta)}.
\]
Now, since  $\frac{\AA(\gamma)}{\mmcz(\gamma)}> \frac{\AA(\delta)}{\mmcz(\delta)}$
we choose $C>0$ so that for all $R\ge C$ one has  
	\[
		\frac{R-n+1}{R+n}\,\frac{\A(\gamma)}{\mmcz(\gamma)}\geq\frac{\A(\delta)}{\mmcz(\delta)}.
	\]
If $\kappa_0\in\N$ is such that such that $2\kappa_0+n+1\ge C$, then, for any $\kappa\geq \kappa_0$, whenever  $\mu_\CZ(\gamma^k)=2\kappa+n+1$ and $\mu_\CZ(\delta^\ell)\in\{\mcz(\gamma^k)-1,\mcz(\gamma^k)+1\}$ for some $k,\ell\in\N$ we have 
\begin{equation}\label{eq:action ineq2}
	\AA(\gamma^k)>\AA(\delta^\ell).
\end{equation}
Since $SH^{S^1,+}_*$ is $\Q$ for $*\in2\N+n-1$ and 0 for $*\in 2\Z+n$, all high good iterates of  $\delta$ must be killed by good iterates of $\gamma$ due to  \eqref{eq:action ineq2}. Since the action decreases along the differential  $\p$ the equation \eqref{eq:action ineq2} implies
\begin{equation}\label{eq:consecutive orbits}
	\#\mu_\CZ^{-1}(2\kappa+n)+1=\#\mu_\CZ^{-1}(2\kappa+n+1),\quad \kappa\geq\kappa_0
\end{equation}
and  
\begin{equation}\label{eq:vanishing of odd differentials}
	\p_{2\kappa+n+2}:SC^{S^1,+}_{2\kappa+n+2}\stackrel{0}{\pf} SC^{S^1,+}_{2\kappa+n+1},\quad \kappa\geq\kappa_0.
\end{equation}
where  $SC^{S^1,+}$ is the chain complex spanned by the (unparametrized) good periodic orbits of period at most $K>>0$ and $\p$ is the differential.
Since $\hat\mu_\CZ(\gamma)>0$ and $\hat\mu_\CZ(\delta)>0$, we can choose $M\in\N$ sufficiently large such that for any $k\geq M$,
\begin{equation}\label{eq1}
	\mu_\CZ(\gamma^k)>2\kappa_0+n+3+2(n-1) \quad \textrm{ and }\quad \mu_\CZ(\delta^k)>2\kappa_0+n+3+2(n-1) .
\end{equation}
According to Theorem \ref{thmCIJ+}, we can find $(N,m_\gamma,m_\delta)\in\N^3$ with $N \geq \kappa_0+n$ satisfying
\begin{equation}\label{eq2}
	\mu_\CZ(\gamma^{2m_\gamma-m})=2N-\mu_\CZ(\gamma^m),\quad \mu_\CZ(\gamma^{2m_\gamma+m})=2N+\mu_\CZ(\gamma^m),\quad 1\leq m\leq M
\end{equation}
and 
\begin{equation}\label{eq3}
	\mu_\CZ(\delta^{2m_\delta-m})=2N-\mu_\CZ(\delta^m),\quad \mu_\CZ(\delta^{2m_\delta+m})=2N+\mu_\CZ(\delta^m),\quad 1\leq m\leq M.
\end{equation}
Using \eqref{eq:iteration and mean index}, we have $\mcz(\gamma^{k+i})-\mcz(\gamma^{k})>-2(n-1)$ for any $k,i\in\N$ because
$\mcz(\gamma^{k+i})-\mcz(\gamma^{k})=\mu_\CZ(\gamma^{k+i})-(k+i)\,\wh\mu_\CZ(\gamma)+i\,\wh\mu_\CZ(\gamma)-\bigl(\mu_\CZ(\gamma^{k})-k\,\wh\mu_\CZ(\gamma)\bigr)$. In particular, for any $m'\geq M$ , $\mcz(\gamma^{2m_\gamma-m'})<\mcz(\gamma^{2m_\gamma-M})+2(n-1)$ and 
$\mcz(\gamma^{2m_\gamma+m'})>\mcz(\gamma^{2m_\gamma+M})-2(n-1)$. Equations \eqref{eq1} and \eqref{eq2}  yield that for any $m'\geq M$, 
	\begin{equation}\label{eq4}
		\mcz(\gamma^{2m_\gamma-m'})\leq 2N-n+1,\quad \mcz(\gamma^{2m_\gamma+m'})\geq 2N+2\kappa_0+n+3.
	\end{equation}
One could deduce a better estimate for $\mcz(\gamma^{2m_\gamma-m'})$ but the estimate mentioned is enough for the proof.
The same holds for $\delta$: for any $m'\geq M$, 
	\begin{equation}\label{eq5}
		\mcz(\delta^{2m_\delta-m'})\leq 2N-n+1,\quad \mcz(\delta^{2m_\delta+m'})\geq 2N+2\kappa_0+n+3.
	\end{equation}
From Theorem \ref{thmCIJ+}, we also know
	\begin{equation}\label{eq6}
		\mcz(\gamma^{2m_\gamma})\leq 2N+n-1,\quad 
		\mcz(\delta^{2m_\delta})\leq 2N+n-1.
	\end{equation}
Moreover, the fact that both $\mmcz(\gamma)$ and $\mmcz(\delta)$ are positive together with \eqref{eq:iteration and mean index} imply that for all $k\in\N$, $\mcz(\gamma^k)$ and $\mcz(\delta^k)$ are bigger than $-n+1$ and therefore
	\begin{equation}\label{eq8}
		\mcz(\gamma^{2m_\gamma-m})<2N+n-1,\quad \mcz(\delta^{2m_\delta-m})<2N+n-1.
	\end{equation}
for all $1\leq m\leq M$ due to \eqref{eq2} and \eqref{eq3}. From \eqref{eq4}, \eqref{eq5}, \eqref{eq6},  and  \eqref{eq8}, we deduce that if 
\[
\mcz(\gamma^k),\,\mcz(\delta^\ell)\in[2N+n,2N+2\kappa_0+n+2],
\]
then $k,\ell\in\N$ are of the form 
\[
k=2m_\gamma+m,\quad \ell=2m_\delta+\tilde{m} \quad\textrm{for some } 1\leq m, \tilde{m}\leq M.
\]
Hence $\mcz(\gamma^k) =2N+r$ with $ n\leq r\leq 2\kappa_0+n+2$  implies $k=2m_\gamma+m$ for some  $1\leq m \leq M$, hence $\mcz(\gamma^m)=r$. Reciprocally if $\mcz(\gamma^{k'})=r$ then $r\le M$ by \eqref{eq1} so that $\mcz(\gamma^{2m_\gamma+k'}) =2N+r$.
The same is true for the indices if the iterates of $\delta$. Hence
\[
	 \#\mu_\CZ^{-1}(r)=\#\mu_\CZ^{-1}(2N+r),\quad  n\leq r\leq 2\kappa_0+n+2.
\]
 Since $N>\kappa_0+n$ we use equation \eqref{eq:consecutive orbits} with  $r=n-1+2q$ and the above to obtain 
\begin{equation}\label{eq:number of good orbits}
	\#\mu_\CZ^{-1}(n-2+2q)+1=\#\mu_\CZ^{-1}(n-1+2q),\quad 1\le q\le \kappa_0+1.
\end{equation}
This implies that the differential $\p_{n}:SC_{n}^{S^1,+}\to SC_{n-1}^{S^1,+}$ vanishes.
Indeed if this were not true, the differential $\p_{*}$ would be nonzero for all $*=2q+n, q\le \kappa_0+1$  to obtain the homology results of Theorem \ref{thm:computing}, in view of \eqref{eq:number of good orbits}.
This would contradicts \eqref{eq:vanishing of odd differentials}.
This implies that 
	\[
		\big(SC_*^{S^1,+},\p_*\big)_{*\in I},\quad I=\Z\cap [-n+3,n-1]
	\] 
is a chain complex with zero homology in view of Theorem \ref{thm:computing} again. We claim that 
 this is impossible by showing that
	\[
		\sum_{q\in (2\Z+{n+1})\cap I}\#\mcz^{-1}(q)>\sum_{q\in (2\Z+{n})\cap I}\#\mcz^{-1}(q).
	\]
Observe from \eqref{eq4}  that  $\mcz(\gamma^k)=2N+q$ with $q \in [0,n-1]$ implies that $k=2m_\gamma+k_0$ with $-M\le k_0\le M$; and, 
by \eqref{eq2},  $q=\mcz(\gamma^{k_0})$ if $k_0>0$ and $q=-\mcz(\gamma^{\vert k_0 \vert})$ if $k_0<0$.
Conversely, if $\mcz(\gamma^{k_0})$ is $q$ (or $-q$) for some $k_0\in\N$, then $\mcz(\gamma^{2m_\gamma+k_0})$ (or $\mcz(\gamma^{2m_\gamma-k_0}))$ is $2N+q$. The same holds for $\delta$.
Hence
	\[
		\#\big(\mu_\CZ^{-1}(-q)\cup\mu_\CZ^{-1}(q)\big)=\#\mu_\CZ^{-1}(2N+q),\quad 0\leq q\leq n-1
	\]
except in the case where $\gamma^{2m_\gamma}$ or $\delta^{2m_\delta}$ is good and has index $2N+q$. Therefore we have
\begin{equation}\label{eq:e_gamma}
		e_\gamma+\sum_{q\in (2\Z+{n+1})\cap I}\#\mcz^{-1}(q)=\sum_{q\in (2\Z+{n+1})\cap [0,n-1]}\#\mcz^{-1}(2N+q)
\end{equation}
where $e_\gamma=1$ if $\gamma^{2m_\gamma}$ is good and otherwise $e_\gamma=0$. We set $e_\delta\in\{0,1\}$ in the same way and have
\begin{equation}\label{eq:e_delta}
		e_\delta+\sum_{q\in (2\Z+{n})\cap I}\#\mcz^{-1}(q)=\sum_{q\in (2\Z+{n})\cap [0,n-2]}\#\mcz^{-1}(2N+q)
\end{equation}
Since we have assumed that $n\ge 3$,  $\#((2\Z+n+1)\cap I)\geq 2$. Therefore using \eqref{eq:consecutive orbits}, \eqref{eq:e_gamma}, and \eqref{eq:e_delta} we deduce
\begin{eqnarray*}
\sum_{q\in (2\Z+{n+1})\cap I}\#\mcz^{-1}(q)&=&\sum_{q\in (2\Z+{n+1})\cap [0,n-1]}\#\mcz^{-1}(2N+q)-e_\gamma\\
&\geq&\sum_{q\in (2\Z+{n})\cap [0,n-2]}\#\mcz^{-1}(2N+q)+\#((2\Z+{n+1})\cap [0,n-1])-e_\gamma\\
&\geq& \sum_{q\in (2\Z+{n})\cap I}\#\mcz^{-1}(q)+2+e_\delta-e_\gamma>\sum_{q\in (2\Z+{n})\cap I}\#\mcz^{-1}(q).
\end{eqnarray*}
This proves the claim and hence the first case.

%%%%%%%%%%%%%%%%%%%%%%%%
\subsection{Second case: $\frac{\A(\gamma)}{\mmcz(\gamma)}<\frac{\A(\delta)}{\mmcz(\delta)}$}
We derive a contradiction in a similar manner to the first case. In the same way as before, the condition $\frac{\A(\gamma)}{\mmcz(\gamma)}<\frac{\A(\delta)}{\mmcz(\delta)}$ implies that  there is $\kappa_0\in\N$ such that  for any $\kappa\geq \kappa_0$, if $\mcz(\gamma^k)\geq2\kappa+n+1$ and  $\mu_\CZ(\delta^\ell)\in\{\mcz(\gamma^k)-1,\mcz(\gamma^k)+1\}$ for some $k,\ell\in\N$, then
\[
	\AA(\gamma^k)<\AA(\delta^\ell).
\]
As before, this implies that 
\begin{equation}\label{conseqcase2}
	\#\mu_\CZ^{-1}(2\kappa+n+1)=\#\mu_\CZ^{-1}(2\kappa+n+2)+1,\quad \kappa\geq\kappa_0
\end{equation}
and  
\begin{equation}\label{conseq2case2}
	\p_{2\kappa+n+1}:SC^{S^1,+}_{2\kappa+n+1}\stackrel{0}{\pf} SC^{S^1,+}_{2\kappa+n},\quad \kappa\geq\kappa_0.
\end{equation}
We choose $(N,m_\gamma,m_\delta)\in\N^3$ to satisfy \eqref{eq1} and \eqref{eq2} as in case one to obtain as before
\[
	 \#\mu_\CZ^{-1}(r)=\#\mu_\CZ^{-1}(2N+r),\quad  n\leq r\leq 2\kappa_0+n+2.
\]
and with  $N>\kappa_0+n$ we use equation \eqref{conseqcase2}    and the above to obtain  the counterpart of \eqref{eq:number of good orbits}
\begin{equation}\label{eq:number of good orbits2}
	\#\mu_\CZ^{-1}(n+2q)+1=\#\mu_\CZ^{-1}(n-1+2q),\quad 1\le q\le \kappa_0+1.
\end{equation}
This implies that the differential $\p_{n+1}:SC_{n+1}^{S^1,+}\to SC_{n}^{S^1,+}$ vanishes. Indeed if it did not vanish, then by induction, using
\eqref{eq:number of good orbits2}, all $\p_{n-1+2q}$ would not vanish for  $1\le q\le \kappa_0+1$  and this would contradict \eqref{conseq2case2}.
Therefore the chain complex
	\[
		\big(SC^{S^1,+}_*,\p_*)_{*\in I},\quad I'=\Z\cap[-n+3,n]
	\]
has vanishing homology. However this is impossible if $n$ is odd.  Indeed,  \eqref{eq:e_gamma}  and  \eqref{eq:e_delta} become, with the same notation,
\[
		e_\gamma+\sum_{q\in (2\Z+{n+1})\cap I'}\#\mcz^{-1}(q)=\sum_{q\in (2\Z+{n+1})\cap [0,n-1]}\#\mcz^{-1}(2N+q)
\]
\[
		e_\delta+\sum_{q\in (2\Z+{n})\cap I'}\#\mcz^{-1}(q)=\sum_{q\in (2\Z+{n})\cap [0,n]}\#\mcz^{-1}(2N+q)
\]
where $e_\gamma=1$ if $\gamma^{2m_\gamma}$ is good and otherwise $e_\gamma=0$ and similarly for  $e_\delta\in\{0,1\}$.
We now use \eqref{conseqcase2} and  get
\begin{eqnarray*}
\sum_{q\in (2\Z+{n+1})\cap I'}\#\mcz^{-1}(q)&=&\sum_{q\in (2\Z+{n+1})\cap [0,n-1]}\#\mcz^{-1}(2N+q)-e_\gamma\\
&=&\sum_{q\in (2\Z+{n})\cap [1,n]}\#\mcz^{-1}(2N+q)+\#((2\Z+{n+1})\cap [0,n-1])-e_\gamma\\
&\geq& \sum_{q\in (2\Z+{n})\cap I'}\#\mcz^{-1}(q)+2 -\alpha +e_\delta-e_\gamma>\sum_{q\in (2\Z+{n})\cap I'}\#\mcz^{-1}(q)-\alpha.
\end{eqnarray*}
Where $\alpha=\#\mcz^{-1}(0)$ if n is even and $\alpha=0$ if $n$ is odd.
This proves the second case when $n$ is odd,
and hence finishes the proof of Proposition \ref{thm:third periodic orbit}.

%%%%%%%%%%%%%%%%%%%%%%%%%%%%%%%%%%%%%%%%%%%%%%%%%%%%%%%%%%%%%%%%%%%%%%%%%%%%%%%%%%%%%%%%%%%%%%%%%%%%%%%%%%%%%%%%%%%%%%%%%%%%%%%%%%%%%%%%%%%%%%%%
\bibliographystyle{alpha}
%\nocite{*}

\bibliography{biblithese.bib}

\begin{thebibliography}{GHHM13}

\bibitem[AD10]{AD}
Mich{{\`e}}le Audin and Mihai Damian.
\newblock {\em Th{\'e}orie de {M}orse et homologie de {F}loer}.
\newblock Savoirs Actuels (Les Ulis). [Current Scholarship (Les Ulis)]. EDP
  Sciences, Les Ulis; CNRS {\'E}ditions, Paris, 2010.

\bibitem[BCE07]{BCE}
Fr{{\'e}}d{{\'e}}ric Bourgeois, Kai Cieliebak, and Tobias Ekholm.
\newblock A note on {R}eeb dynamics on the tight 3-sphere.
\newblock {\em J. Mod. Dyn.}, 1(4):597--613, 2007.

\bibitem[BLMR85]{blmr}
Henri Berestycki, Jean-Michel Lasry, Giovanni Mancini, and Bernhard Ruf.
\newblock Existence of multiple periodic orbits on star-shaped {H}amiltonian
  surfaces.
\newblock {\em Comm. Pure Appl. Math.}, 38(3):253--289, 1985.

\bibitem[BO10]{BOjems}
Fr{{\'e}}d{{\'e}}ric Bourgeois and Alexandru Oancea.
\newblock Fredholm theory and transversality for the parametrized and for the
  {$S^1$}-invariant symplectic action.
\newblock {\em J. Eur. Math. Soc. (JEMS)}, 12(5):1181--1229, 2010.

\bibitem[BO12]{bo}
Fr{{\'e}}d{{\'e}}ric Bourgeois and Alexandru Oancea.
\newblock {$S^1$}-equivariant symplectic homology and linearized contact
  homology.
\newblock arXiv:1212.3731, 2012.

\bibitem[BO13]{BOind}
Fr{{\'e}}d{{\'e}}ric Bourgeois and Alexandru Oancea.
\newblock The index of {F}loer moduli problems for parametrized action
  functionals.
\newblock {\em Geom. Dedicata}, 165:5--24, 2013.

\bibitem[CGH12]{GH}
Daniel Cristofaro-Gardiner and Michael Hutchings.
\newblock From one {R}eeb orbit to two.
\newblock {\em Preprint arXiv:1202.4839, to appear in J. Diff. Geom.}, 2012.

\bibitem[CZ84]{CZ}
Charles Conley and Eduard Zehnder.
\newblock Morse-type index theory for flows and periodic solutions for
  {H}amiltonian equations.
\newblock {\em Comm. Pure Appl. Math.}, 37(2):207--253, 1984.

\bibitem[EH87]{EH87}
I.~Ekeland and H.~Hofer.
\newblock Convex {H}amiltonian energy surfaces and their periodic trajectories.
\newblock {\em Comm. Math. Phys.}, 113(3):419--469, 1987.

\bibitem[Eke90]{E}
Ivar Ekeland.
\newblock {\em Convexity methods in {H}amiltonian mechanics}, volume~19 of {\em
  Ergebnisse der Mathematik und ihrer Grenzgebiete (3) [Results in Mathematics
  and Related Areas (3)]}.
\newblock Springer-Verlag, Berlin, 1990.

\bibitem[EL80]{EL}
Ivar Ekeland and Jean-Michel Lasry.
\newblock On the number of periodic trajectories for a {H}amiltonian flow on a
  convex energy surface.
\newblock {\em Ann. of Math. (2)}, 112(2):283--319, 1980.

\bibitem[GG15]{GGo}
Viktor~L. Ginzburg and Yusuf G{{\"o}}ren.
\newblock Iterated index and the mean euler characterstic.
\newblock {\em J. Topol. Anal.}, 7:453--481, 2015.

\bibitem[GHHM13]{GHHM}
Viktor~L. Ginzburg, Doris Hein, Umberto~L. Hryniewicz, and Leonardo Macarini.
\newblock Closed {R}eeb orbits on the sphere and symplectically degenerate
  maxima.
\newblock {\em Acta Math. Vietnam.}, 38(1):55--78, 2013.

\bibitem[G{\"{u}}r15]{Gurel}
Ba{\c{s}}ak~Z. G{\"{u}}rel.
\newblock Perfect {R}eeb flows and action-index relations.
\newblock {\em Geom. Dedicata}, 174:105--120, 2015.

\bibitem[Gut14]{gutt2}
Jean Gutt.
\newblock Generalized {C}onley-{Z}ehnder index.
\newblock {\em Annales de la facult{\'e} des sciences de Toulouse},
  23(4):907--932, 2014.

\bibitem[Gut15]{gutt}
Jean Gutt.
\newblock The positive equivariant symplectic homology as an invariant for some
  contact manifolds.
\newblock {\em Preprint arXiv:1503.01443}, 2015.

\bibitem[HT09]{HT}
Michael Hutchings and Clifford~Henry Taubes.
\newblock The {W}einstein conjecture for stable {H}amiltonian structures.
\newblock {\em Geom. Topol.}, 13(2):901--941, 2009.

\bibitem[HWZ98]{HWZ2}
H.~Hofer, K.~Wysocki, and E.~Zehnder.
\newblock The dynamics on three-dimensional strictly convex energy surfaces.
\newblock {\em Ann. of Math. (2)}, 148(1):197--289, 1998.

\bibitem[HWZ03]{HWZ3}
H.~Hofer, K.~Wysocki, and E.~Zehnder.
\newblock Finite energy foliations of tight three-spheres and {H}amiltonian
  dynamics.
\newblock {\em Ann. of Math. (2)}, 157(1):125--255, 2003.

\bibitem[Kan13]{K}
Jungsoo Kang.
\newblock Equivariant symplectic homology and multiple closed reeb orbits.
\newblock {\em Internat. J. Math.}, 24(13), 2013.

\bibitem[LL14]{LL2}
Hui Liu and Yiming Long.
\newblock The existence of two closed characteristics on every compact
  star-shaped hypersurface in $\mathbb{R}^4$.
\newblock {\em Acta Mathematica Sinica, English Series}, Published online: DOI:
  10.1007/s10114-014-4108-1, 2014.

\bibitem[LLWZ14]{LLWZ}
Hui Liu, Yiming Long, Wei Wang, and Ping'an Zhang.
\newblock Symmetric closed characteristics on symmetric compact convex
  hypersurfaces in {$\bold{R}^8$}.
\newblock {\em Commun. Math. Stat.}, 2(3-4):393--411, 2014.

\bibitem[Lon02]{Lon02}
Yiming Long.
\newblock {\em Index theory for symplectic paths with applications}, volume 207
  of {\em Progress in Mathematics}.
\newblock Birkh{\"a}user Verlag, Basel, 2002.

\bibitem[LZ02]{LZ}
Yiming Long and Chaofeng Zhu.
\newblock Closed characteristics on compact convex hypersurfaces in {$\bold
  R^{2n}$}.
\newblock {\em Ann. of Math. (2)}, 155(2):317--368, 2002.

\bibitem[Rab79]{Rab79}
Paul~H. Rabinowitz.
\newblock Periodic solutions of a {H}amiltonian system on a prescribed energy
  surface.
\newblock {\em J. Differential Equations}, 33(3):336--352, 1979.

\bibitem[Sal99]{Sal}
Dietmar Salamon.
\newblock Lectures on {F}loer homology.
\newblock In {\em Symplectic geometry and topology ({P}ark {C}ity, {UT},
  1997)}, volume~7 of {\em IAS/Park City Math. Ser.}, pages 143--229. Amer.
  Math. Soc., Providence, RI, 1999.

\bibitem[Sei08]{Seidel}
Paul Seidel.
\newblock A biased view of symplectic cohomology.
\newblock In {\em Current developments in mathematics, 2006}, pages 211--253.
  Int. Press, Somerville, MA, 2008.

\bibitem[SZ92]{SalamonZehnder}
Dietmar Salamon and Eduard Zehnder.
\newblock Morse theory for periodic solutions of {H}amiltonian systems and the
  {M}aslov index.
\newblock {\em Comm. Pure Appl. Math.}, 45(10):1303--1360, 1992.

\bibitem[Vit99]{V}
C.~Viterbo.
\newblock Functors and computations in {F}loer homology with applications. {I}.
\newblock {\em Geom. Funct. Anal.}, 9(5):985--1033, 1999.

\bibitem[Wan09]{Wang09}
Wei Wang.
\newblock Symmetric closed characteristics on symmetric compact convex
  hypersurfaces in {$\bold R^{2n}$}.
\newblock {\em J. Differential Equations}, 246(11):4322--4331, 2009.

\bibitem[Wan12]{Wang}
Wei Wang.
\newblock On a conjecture of {A}nosov.
\newblock {\em Adv. Math.}, 230(4-6):1597--1617, 2012.

\bibitem[Wan13]{Wang13}
Wei Wang.
\newblock Closed characteristics on compact convex hypersurfaces in
  $\mathbb{R}^8$.
\newblock {\em Preprint arXiv:1305.4680}, 2013.

\bibitem[WHL07]{WHL}
Wei Wang, Xijun Hu, and Yiming Long.
\newblock Resonance identity, stability, and multiplicity of closed
  characteristics on compact convex hypersurfaces.
\newblock {\em Duke Math. J.}, 139(3):411--462, 2007.

\end{thebibliography}

%%%%%%%%%%%%%%%%%%%%%%%%%%%%%%%%%%%%%%%%%%%%%%%%%%%%%%%%%%%%%%%%%%%%%%%%%%%%%%%%%%%%%%%%%%%%%%%%%%%%%%%%%%%%%%%%%%%%%%%%%%%%%%%%%%%%%%%%%%%%%%%%
\end{document}